\documentclass[12pt]{amsart}

\usepackage[numbers,sort&compress]{natbib}
\usepackage{amssymb,amsthm,amsmath}

\title[ ] {A Separation lemma on sub-lattices}
\author
        {W.-M.~Wang}
\address{CNRS and D\'epartement de Math\'ematique \\Cergy Paris Universit\'e \\95302 Cergy-Pontoise Cedex, France}
\email{wei-min.wang@math.cnrs.fr}

 \theoremstyle{plain}
\newtheorem{theorem}{Theorem}[section]

\newtheorem{lemma}[theorem]{Lemma}

\newtheorem*{lemman}{New Separation Lemma}
\newtheorem*{lemmaw}{Small Scale Separation Lemma}
\newtheorem*{corollaryn}{Corollary}
\newtheorem*{lemmam}{Main Lemma}

\theoremstyle{definition}

\newtheorem{remark}[theorem]{Remark}

\numberwithin{equation}{section}

\begin{document}
\date{}
\begin{abstract}
We prove that Bourgain's separation lemma, Lemma~20.14 \cite{B2} holds 
at {\it fixed} frequencies and their neighborhoods, on sub-lattices, sub-modules
of the {\it dual lattice} associated with a quasi-periodic Fourier series in two dimensions.
And by extension holds on the affine spaces.
Previously Bourgain's lemma was not deterministic, and is valid only for a set of frequencies of positive measure. 
The new separation lemma generalizes classical lattice partition-type results to the hyperbolic Lorentzian setting,
with signature $(1, -1, -1)$, and could be of independent interest.

Combined with the method in \cite{W2}, this
should lead to the existence of quasi-periodic solutions
to the nonlinear Klein-Gordon equation with the usual polynomial nonlinear term
$u^{p+1}$.

 \end{abstract}

\maketitle
\section {Introduction and the new separation lemma}   
We consider the nonlinear Klein-Gordon equation (NLKG) on the $d$-torus $\mathbb T^d=[0, 2\pi)^d$:
\begin{equation}\label{(1.1)}
\frac{\partial^2u}{\partial t^2}-\Delta u+u+u^{p+1}=0,
\end{equation}
where $\Delta$ is the Laplacian:
$$\Delta u=\sum_{i=1}^d\frac{\partial^2 u}{\partial x_i^2};$$
$p\in\mathbb N$ and is arbitrary; considered as a function on $\mathbb R^d$, $u$ is periodic, satisfying  
$u(\cdot, x)=u(\cdot, x+2j\pi)$, for all $j\in\mathbb Z^d$. We are interested in solutions which depend quasi-periodically,
``periodic" with several frequencies on time in the form of a (convergent) series:
\begin{equation}\label{1.2}
u(t, x)=\sum_{(n,j)\in\Bbb Z^b\times \Bbb Z^d}a(n, j)e^{ i({n\cdot\omega t}+j\cdot x)},
\end{equation}
where $b$ corresponds to the number of base frequencies in time, $\cdot$ is the usual scalar product, and
$\omega\in\mathbb R^b$ is the frequency.
It is convenient to think of $u$ as defined on the torus $\mathbb T^b\times\mathbb T^d$.  
Series of the form \eqref{1.2} form a closed set under multiplication. So $u^{p+1}$ is also of this form. 
\smallskip

Below we write $\mathbb Z^{b+d}$ for $\mathbb Z^b\times \mathbb Z^d$,
seen as the {\it dual lattice} to $\mathbb T^b\times \mathbb T^d$.
To simplify language, we will often say that 
$\mathbb Z^{b+d}$ is the lattice, or module.
Our goal in this paper
is to derive geometric properties of 
(1.2) under the action of the linear operator: 
\begin{equation}\label{1.3}
L:=\frac{\partial^2}{\partial t^2}-\Delta + 1.
\end{equation}

The linear operator $L$ is separable. Define the wave operator $D$: 
\begin{equation}\label{1.4}
D:=\sqrt{-\Delta+1}.
\end{equation} 
Its spectrum: 
$$\sigma (D)=\{\sqrt{|j|^2+1}\,|\,j\in\Bbb Z^d\},$$
by using the Fourier series in (1.2) and setting $t=0$.
In $d\geq 2$, the spectrum is degenerate, and the spacing of non-equal eigenvalues tends to zero. 

Operating on the full quasi-periodic Fourier series (1.2), $L$ leads to a diagonal matrix $P$ on $\mathbb Z^{b+d}$
with matrix elements: 
\begin{equation}\label{1.5}
P(n, j):=P((n, j), (n, j))=(n\cdot\omega)^2-j^2-1, \, (n, j)\in\mathbb Z^{b+d},
\end{equation}
where for simplicity, $j^2$ stands for $|j|^2$. 
The set of $(n, j)$ such that $|P(n, j)|<1$ plays an essential role in solving (1.1).  In particular,
since $P=0$ solves the linear equation:
\begin{equation}\label{1.6}
\frac{\partial^2u}{\partial t^2}-\Delta u+u=0,
\end{equation}
we shall start from the linear solutions. 

Denote linear solutions by $u^{(0)}$, and let
\begin{equation}\label{1.7}
u^{(0)}(t, x)=\sum_{k=1}^b a_k e^ {i(-(\sqrt{j_k^2+1})t+j_k\cdot x)},
\end{equation}
where $j_k\in\mathbb Z^d\backslash\{0\}$, $k=1, 2, ..., b$, be such a solution. We may write $u^{(0)}$ in the form (1.2):
$$
\aligned u^{(0)}(t, x)&=\sum_{k=1}^b a_k e^{i (-(\sqrt{j_k^2+1})t+j_k\cdot x)}\\
:&=\sum_{k=1}^b a(-e_{k},  j_k)e^{i(-e_k\cdot\omega^{(0)}t+j_k\cdot x)}, \endaligned
$$
where $e_{k}$, $k=1, 2, ..., b$, are the canonical basis of $\mathbb Z^b$,
$$a(-e_{k}, j_k)=a_k, $$ and
$$\omega^{(0)}:=\{\sqrt{j_k^2+1}\}_{k=1}^{b},\, (j_k\neq 0).$$


We are interested in small solutions to (1.1). It is convenient to add $\delta$, $\delta\ll 1$ to the nonlinear term and write
the NLKG as 
\begin{equation}\label{1.8}
\frac{\partial^2u}{\partial t^2}-\Delta u+u+\delta u^{p+1}=0.
\end{equation}
We seek $u$ close to  $u^{(0)}$.
From the structure of the equations (1.7) and \eqref{1.8}, 
the sub-lattice $\mathcal U\subset \mathbb Z^{b+d}$, generated by
the set $\{(-e_k, j_k)\}_{k=1}^b$ plays a special role. We call it 
the {\it sub-module generated by  $u^{(0)}$}.
Writing the left side of \eqref{1.8} as $F(u)$, this is also in view 
of the (functional) derivative: 
$$F'(u)= \frac{\partial^2}{\partial t^2}-\Delta +1+\delta (p+1)u^{p},$$
which plays a central role in the nonlinear analysis, cf. \cite{W2}\cite{W1}. 
\smallskip

\noindent{\it Remark.} 
In the process of solving (1.8), $\omega^{(0)}$ is replaced by $\omega$, which is continuously being modulated,
but it {\it does not} change the dual lattice $\mathbb Z^{b+d}$, nor the sub-lattice, sub-module $\mathcal U\subset \mathbb Z^{b+d}$.

\begin{lemma} On the sub-module, sub-lattice $\mathcal U\subset \mathbb Z^{b+d}$, generated by
the set $\{(-e_k, j_k)\}_{k=1}^b$, if $(\nu, \eta)\in\mathcal U$, and $\nu=0$, then $\eta=0$.  
\end{lemma}\label{1.1}
\begin{proof} This follows from the property of the generating set:  
$\{(-e_k, j_k)\}_{k=1}^b$. Let $$\nu=\sum_{k=1}^bC_ke_k.$$
If $\nu=0$, then $C_k=0$ for $k=1, 2, ..., b$. So 
$$\eta=-\sum_{k=1}^bC_kj_k=0.$$ 
\end{proof} 

\noindent{\it Remark.} The converse is generally not true, since in general, $b>d$. 

\subsection{The new separation lemma}
Denote by $\pi_k$ the set of primes ($\neq 1$) that occur in odd powers in the prime decomposition of
$j_k^2+1$, $j_k\neq 0$, $k=1, 2, ..., b$. 
Assume that 

\smallskip

\noindent {(*)} $\pi_k\neq\emptyset$, $k=1, 2, ..., b$, and $\pi_k\not\subseteq \pi_{k'}$ for $k\neq k'$.
 
Let $S$ be the set of $j_1, j_2, ..., j_b\in\mathbb Z^d$ such that (*) is satisfied, and $S'$ the set of $j_1, j_2, ..., j_b\in\mathbb Z^d$ used in \cite{W2}, i.e., such that 
 $$1< j_1^2+1<j_2^2+1<\cdots.<j_{b}^2+1,$$
and square-free, then we have
$$S\supset S'.$$ Recall that $S'$ has positive density from \cite{LX}, so $S$ has (at least) positive density as well.

\smallskip

Our main result is the following: 

\begin{lemman}
Let $d=2$, and 
$$u^{(0)}(t, x)=\sum_{k=1}^b a_k e^ {i(-(\sqrt{j_k^2+1})t+j_k\cdot x)},$$
be a solution to the linear equation \eqref{1.6}.
Assume  that 
(*) is satisfied.
Let $\delta\in (0, 1]$ and $B>1$ be an integer. 
Set $$\omega=\omega^{(0)}+\delta \omega'\in\mathbb R^b,$$
with $|\omega'|\leq1$. There exists a set $\mathcal A$, $0\in\mathcal A$, in $\omega'$, whose
complement is of measure less than $\delta$, such that on $\mathcal A$, 
the following holds:

Let $\theta\in\mathbb R$, and set $x:=(n, h)\in\mathbb Z^{b+2}$. Let $\mathcal U\subset \mathbb Z^{b+2}$ be 
the sub-module generated by $u^{(0)}$. Consider a sequence, a subset  $\{x_k\}_{k=0}^\ell\subset\mathcal U\subset\mathbb Z^{b+2}$,
satisfying 
\begin{equation}\label {5.15}
|x_i-x_{i-1}|<B, \quad 1\leq i\leq \ell.
\end{equation}
Assume that 
for all $i$, $1\leq i\leq \ell$,
\begin{equation}\label{5.16}
|(n_i\cdot\omega+\theta)^2-h_i^2-1|<1,
\end{equation}
and that there exists $B'>1$, such that 
\begin{equation}\label{5.17}
\text {max}_{m\in\mathbb Z^2}|\{k|h_k=m; 0\leq k\leq \ell\}|< B'.\end{equation}
There exists $B_0>0$,  such that if $B>B_0$, then there exists
$C>1$, and
\begin{equation}\label{5.18}
\ell<(BB')^{C},\end{equation}
for all fixed $\theta$.
\end{lemman}
 
We note that the lemma holds {\it at} $\omega=\omega^{(0)}$, and the
separation property is {\it stable} near $\omega^{(0)}$.
This refines Bourgain's separation lemma,  Lemma 20.14 \cite{B2}, which uses
$\omega$ as parameters, in two dimensions. The new Lemma generalizes the classical lattice partition-type
results to the hyperbolic Lorentzian setting with signature $(1, -1, -1)$, at irrational square root frequencies,
and is essential to the proof of quasi-periodic solutions to (1.1). 
\smallskip 

\begin{corollaryn} The New Separation Lemma holds on the affine spaces $\alpha+\mathcal U$ of the module $\mathbb Z^{b+2}$, for all
$\alpha\in\mathbb Z^{b+2}$, i.e., the conclusion of the Lemma holds with $\alpha+\mathcal U$ replacing $\mathcal U$, for all $\alpha$. 
\end{corollaryn}

This is because the proof of the Lemma relies only on the differences $x_i-x_{i-1}$, $i=1, 2, ..., \ell$. \hfill$\square$

\smallskip

\subsection{The role of separation lemmas} We note that the New Separation Lemma is {\it asymptotic}, it is applicable for large $B$. (This is also the case
for Lemma~20.14 \cite{B2}.) Chap.~20 \cite{B2} deals with {\it parameter dependent} wave equations,
and asymptotic properties are sufficient for a proof of quasi-periodic solutions. (This is also the case in Chap.~19, for the {\it parameter dependent} nonlinear Schr\"odinger equations (NLS).)

For a {\it fixed} equation, such as (1.1), however, this is not enough, we need precise, {\it non-asymptotic} separation properties as well. 
This is provided by Lemma~5.2 \cite{W2}, which is a separation lemma at small scales, valid in arbitrary dimensions $d$.
(Cf. the reformulation of Lemma~5.2 in sect.~2.3.)  
 The lack of a separation 
lemma at large scales was the only reason why in \cite{W2} we were not able to deal with polynomial nonlinearity.  

We expect that the combination of the New Separation Lemma with Lemma~5.2 \cite{W2}
will now lead to a proof of quasi-periodic solutions to the NLKG in (1.1). 
(Remark that this approach is analogous to that used in \cite{W1} for the fixed NLS.)

\smallskip

Note that quasi-periodic solutions to NLS are known from the works \cite{W1}\cite{PrPr}, cf. also the earlier related works \cite{B1}\cite{B2}\cite{EK},
but the corresponding problem for NLKG in higher dimensions is significantly more difficult due to the shrinking gaps and the degenerate spectra.
In particular, number theory seeps into the arguments.

\smallskip

The New Lemma proves a separation property on a hyperbolic Lorentzian lattice of signature $(1, -1, -1)$.
Since most of the lattice partition-type results are in the elliptic (parabolic) setting, we believe the Lemma is of 
independent interest, aside from its applications to the NLKG in (1.1). So we present it
separately. This is also because it seems to be a good occasion to elucidate some of the algebraic and geometric
structures underlying the approach developed in \cite{W1}\cite{W2}. (See also the review paper \cite{W3}.)

\bigskip

This paper is self-contained, the proofs are mostly algebraic and elementary, and can be read independently, without prior knowledge of NLKG. 

\subsection{Ideas of the proof}
As mentioned above, most of the lattice partition-type results are on Euclidean lattice, 
for example, the separated cluster structure of the $(n_1, n_2, ..., n_d)\in\mathbb Z^d$
on the hypersphere: 
$$n_1^2+n_2^2+...+n_d^2=R.$$
The positive definite signature $(1, 1, ..., 1)$
played a key role.

Lorentzian lattice deprives us of positivity, however, in its place, one may try to
impose certain non-zero conditions on quadratic forms.  Given $\omega\in\mathbb R^b$,
for $\xi=(\nu, \eta)\in\mathbb Z^{b+d}$,  
define the operators $T_\pm$ from $\mathbb R^{b+d}\to \mathbb R^{1+d}$
to be:
\begin{equation*}
T_\pm \xi=T_\pm (\nu, \eta)=(\nu\cdot\omega, \pm\eta).
\end{equation*}
Note that 
$$T_+\xi T_-\xi=(\nu\cdot\omega)^2-\eta^2,$$
and is of a form similar to the left side in \eqref{5.16}. 

In $d=2$, for example, given two vectors 
$0\neq \xi\neq \xi' \in \mathbb Z^{b+2}$, $\text{dim} \{\xi, \xi'\}=2$, one may define
the following $2\times 2$ matrix: 

\begin{equation*}
T_{+-}=\begin{pmatrix} T_+\xi T_-\xi&T_+\xi T_-{\xi'}\\
 T_+{\xi'}T_-\xi&T_+{\xi'}T_-{\xi'}\\
 \end{pmatrix}.
\end{equation*}
Assume $\eta$ or $\eta'\neq 0$, Lemma~20.23 \cite{B2} proves that 
$|\det T_{+-}|$ has a  lower bound 
on a large set of $\omega\in(0, 1]^b$. Similarly for the $1\times 1$ determinants, when $\text{dim} \{\xi, \xi'\}=1$.
More generally, Lemma~20.23 \cite{B2} proves lower bounds on the absolute values of all such determinants, on a large set of $\omega\in(0, 1]^b$, in arbitrary dimensions $d$.

One may view $T_{+-}$ as an analogue of a Gram matrix in the Lorentzian setting. The lower bounds 
permit proof of the separation lemma, Lemma~20.14 \cite{B2}. 

\subsection{The new idea for fixed frequencies and their neighborhoods} The proof of the New Separation Lemma
hinges on proving lower bounds on $|\det T_{+-}|$. The obvious difficulty is that, unlike in Lemma~20.23 \cite{B2}, $\omega=\omega^{(0)}$
is {\it fixed}, and will not budge. However, we observe that $\det T_{+-}$ is of the form
$$D:=\det T_{+-}=a+\sum _{k<k'}b_{kk'}\omega_k\omega_{k'},$$ 
where $a$, $b_{kk'}\in\mathbb Z$. (There is no quartic in $\omega$ term due to the wedge product.)   
So assume (*) holds, $D=0$, if and only if
$a=b_{kk'}=0$, for all pairs, $k$ and $k'$ ($k<k'$). For example, when 
$k, k'\in \{1,2\}$, this leads to two equations; and when $k, k'\in \{1,2,3\}$,
4 equations. In two dimensions, this suffices to make variable reductions to reach a contradiction 
on the sub-module $\mathcal U$ using Lemma~1.1. So $D\neq 0$. The result 
in \cite{Schm1} on simultaneous Diophantine approximations then gives the
desired lower bound. Subsequently a simple argument gives also a lower bound 
in the neighborhood of $\omega^{(0)}$, cf. the linear case in Lemma~5.3 \cite{W2}. 

In three dimensions and above, the above strategy to prove $D\neq 0$
meets a difficulty, namely there are more unknowns
than equations when the $\nu$ have small support. For example,
when $d=3$, for $\nu$ with less than $6$ non-zero components.
For $\nu$ with larger support, the same proof works. Similarly for $d>3$.

\bigskip

We prove the new separation lemma in the next section, and conclude by translating the
separation lemma at small scales, Lemma~5.2 \cite{W2} into the present language.

\section{Proof of the new separation lemma}

Given $\omega\in\mathbb R^b$,
for $\xi=(\nu, \eta)\in\mathbb Z^{b+d}$,  
define the operators $T_\pm$ from $\mathbb R^{b+d}\to \mathbb R^{1+d}$
as mentioned before:
\begin{equation}\label{Tpm}
T_\pm \xi=T_\pm (\nu, \eta)=(\nu\cdot\omega, \pm\eta).
\end{equation}
 Let $\omega^{(0)}_k=\sqrt{j_k^2+1}$, for $k=1, 2, ..., b$, and assume (*) is satisfied. 
Below we specialize to $d=2$.

\subsection{Non-zero generalized Gram determinants at $\omega=\omega^{(0)}$} 
\begin{lemma} For $0\neq\xi=(\nu, \eta)\in\mathcal U\subset \mathbb Z^{b+2}$, when $\omega=\omega^{(0)}$,
\begin{equation}\label{1x1+}
T_+\xi T_+\xi=(\nu\cdot\omega)^2+\eta^2\neq 0;
\end{equation}
and 
\begin{equation}\label{1x1}
T_+\xi T_-\xi=(\nu\cdot\omega)^2-\eta^2\neq 0.
\end{equation}
\end{lemma}

\begin{proof}
Eq. \eqref{1x1+} clearly holds for non-zero $\xi$ when $\eta\neq 0$, and when $\eta=0$, 
using $\nu\neq 0$ and the rational independence of 
$\omega_{k}$, $k=1, 2, ..., b$.
Set $$D= T_+\xi T_-\xi.$$ 
If $\nu$ has two or more non-zero components, $D$ contains 
one or more irrational terms using (*). The rational independence of 
$1$, $\omega_{k}\omega_{k'}$, $k<k'$, $k$, $k'=1, 2, ..., b$, gives 
$$D\neq 0.$$
If $\nu$ has only one non-zero component,  (without loss) say $\nu_1$, then on the sub-module $\eta=\nu_1j_1$.
So $$D=\nu_1^2\omega_1^2-\nu_1^2j_1^2=\nu_1^2\neq 0.$$
\end{proof}

We proceed to the $2\times 2$ matrices. 

\begin{lemma} Assume $\xi=(\nu, \eta), \xi'=(\nu', \eta') \in\mathbb Z^{b+2}$,
$\text{dim }\{\xi, \xi'\}=2$, and $\text{dim }\{\eta, \eta'\}=1$.
When $\omega=\omega^{(0)}$, the $2\times2$ matrices:
\begin{equation}\label{defT++}
T_{++}=\begin{pmatrix} T_+\xi T_+\xi&T_+\xi T_+{\xi'}\\
 T_+{\xi'}T_+\xi&T_+{\xi'}T_+{\xi'}\\
 \end{pmatrix},
\end{equation}
and 
\begin{equation}\label{defT+-}
T_{+-}=\begin{pmatrix} T_+\xi T_-\xi&T_+\xi T_-{\xi'}\\
 T_+{\xi'}T_-\xi&T_+{\xi'}T_-{\xi'}\\
 \end{pmatrix},
\end{equation}
satisfy 
\begin{equation}
\det T_{++}\neq 0,
\end{equation}
and 
\begin{equation}\label{tpm}
\det T_{+-}\neq 0.
\end{equation}
\end{lemma}

\begin{proof} If $\eta$ or $\eta'=0$, without loss, one may assume $\eta'=0$, 
then since $\nu'\neq 0$,
$$\det T_{++}=\eta^2(\nu'\cdot\omega)^2\neq 0,$$
using the rational independence of $1$ and $\omega_{k}\omega_{k'}$, $k<k'$, $k$, $k'= 1, 2, ..., b.$ 
If $\eta\neq 0$ and $\eta'\neq 0$, and
$$\text{dim }\{(\nu\cdot\omega, \eta), (\nu'\cdot\omega, \eta')\}=1,$$
then since $\text{dim }\{\eta, \eta'\}=1$,
$\eta'=c\eta$, $c\neq 0$. So 
$$\nu'\cdot\omega=c\nu\cdot\omega.$$
So $$\nu'=c\nu,$$ using the rational independence of $1$, and $\omega_k$, $k=1, 2, ..., b$.
This contradicts $\text{dim }\{\xi, \xi'\}=2$,
so $$\det T_{++}\neq 0.$$

We now prove \eqref{tpm}.  If $\eta$ or $\eta'=0$, then the same argument as for $T_{++}$ shows
$$\det T_{+-}\neq 0.$$
If $\eta\neq 0$ and $\eta'\neq 0$, then $\eta'=c\eta$, $c\neq 0$.
Using this to compute the determinant, we obtain
$$\det T_{+-}=-\eta^2[c(\nu\cdot\omega)-(\nu'\cdot\omega)]^2.$$
So $\det T_{+-}=0$ if and only if $\nu'=c\nu$. So
$$(\nu', \eta')=c(\nu, \eta),$$
contradicting 
$$\text{dim }\{(\nu, \eta), (\nu', \eta')\}=2.$$
So \eqref{tpm} holds.
\end{proof}

\noindent{\it Remark.} Note that the above lemma holds on $\mathbb Z^{b+2}$, and not just on the sub-module $\mathcal U$. In fact 
this lemma, with obvious modifications, i.e., the $\xi's$ span $d$ dimensions; while the $\eta's$ $(d-1)$ dimensions, holds in any dimension $d$.   

\smallskip

We now come to the crux of the matter, the lemma whose proof uses $d=2$.
\begin{lemmam} Assume $\xi=(\nu, \eta), \xi'=(\nu', \eta') \in\mathcal U\subset\mathbb Z^{b+2}$,
$\text{dim }\{\xi, \xi'\}=2$, and $\text{dim }\{\eta, \eta'\}=2$.
When $\omega=\omega^{(0)}$, the $2\times2$ matrices as defined in \eqref{defT++} and \eqref{defT+-}
satisfy 
\begin{equation}\label{++}
\det T_{++}\neq 0,
\end{equation}
and 
\begin{equation}\label{+-}
\det T_{+-}\neq 0.
\end{equation}
\end{lemmam}

\begin{proof} Since $\text{dim }\{\eta, \eta'\}=2$, \eqref{++} is obviously true. We only need to prove \eqref{+-}.

Given $\nu\in\mathbb Z^b$, we call $\text{supp }\nu$, the set in $i$, $i=1, 2, ..., b$,
such that $\nu_i\neq 0$. Without loss, there are two cases:

\smallskip

\noindent {\bf (A)} $\text{supp }\nu\not\supseteq\text{ supp }\nu';$

and 

\noindent {\bf (B)} $\text{supp }\nu\supseteq\text{ supp }\nu'.$

\bigskip

\noindent {\bf Case (A)}: Direct computation gives
\begin{equation}\label{det}
\det T_{+-}= -{\eta'}^2(\nu\cdot\omega)^2-{\eta}^2(\nu'\cdot\omega)^2+\eta^2{\eta'}^2+2(\nu\cdot\omega)(\nu'\cdot\omega)\eta\cdot\eta'-(\eta\cdot\eta')^2.
\end{equation}
If $\eta\cdot\eta'\neq 0$, then clearly 
$$(\nu\cdot\omega)(\nu'\cdot\omega)$$ contains a square root term not in $(\nu\cdot\omega)^2$ or $(\nu'\cdot\omega)^2$, so \eqref{+-} holds.
If $\eta\cdot\eta'=0$, $|\text{supp }\nu|$ or $|\text{supp }\nu'|\geq 2$, without loss, assume $|\text{supp }\nu|\geq 2$, then $(\nu\cdot\omega)^2$ contains a square root not in $(\nu'\cdot\omega)^2$,
so \eqref{+-} holds. Finally if $|\text{supp }\nu|=1$ and  $|\text{supp }\nu'|=1$, then, without loss, one may assume 
$$\nu=(\nu_1, 0, 0, .., 0), \, \nu_1\neq 0$$ 
and 
$$\nu'=(0, \nu'_2, 0, .., 0), \, \nu'_2\neq 0.$$ 
Using this in \eqref{det} gives
$$\det T_{+-}=-\nu_1^2{\nu_2'}^2(j_1^2j_2^2+j_1^2+j_2^2)\neq 0.$$

\bigskip
We now proceed to case (B), it is further divided into three subcases: 

\smallskip

\noindent {\bf Case (B1)}: $|\text{supp }\nu|=1$.

In this case, $\nu'=c\nu$, so $\eta'=c\eta$ by using Lemma~1.1,
which contradicts 
$\text{dim }\{(\nu, \eta), (\nu', \eta')\}=2$. So \eqref{+-} holds.

\bigskip

\noindent {\bf Case (B2)}: $|\text{supp }\nu|=2$.

Without loss, one may write 
$$\nu=(\nu_1, \nu_2, 0, .., 0), \, \nu_1\neq 0, \nu_2\neq 0,$$ 
and 
$$\nu'=(\nu'_1, \nu'_2, 0, .., 0), \, \nu'_2\neq 0.$$ 

We compute the terms that enter in the matrix $T_{+-}$. Write
$$\sum_{i=1}^b \nu_ij_i=\nu\cdot j,$$
$$x=\nu_1/\nu_2,$$
and 
$$y=\nu'_1/\nu'_2.$$
We have
\begin{equation}\label{nu}
(\nu\cdot\omega)^2-\eta^2=(\nu\cdot\omega)^2-(\nu\cdot j)^2
=\nu_2^2(1+x^2-2xj_1\cdot j_2+2x\omega_1\omega_2),
\end{equation}
and
\begin{equation}\label{nunu'}
(\nu\cdot\omega)(\nu'\cdot\omega)-(\nu\cdot j)(\nu'\cdot j)
=\nu_2\nu'_2[1+xy-(x+y)j_1\cdot j_2+(x+y)\omega_1\omega_2].
\end{equation}
Write $$A=j_1\cdot j_2,$$ and $$\Omega=\omega_1\omega_2.$$
If $\det T_{+-}=0$, then 
\begin{equation}\label{det0}
(1+x^2-2Ax+2x\Omega)(1+y^2-2Ay+2y\Omega)
=[1+xy-(x+y)A+(x+y)\Omega]^2.
\end{equation}
Equate the irrational part gives:
$$y(1+x^2-2Ax)+x(1+y^2-2Ay)=[1+xy-(x+y)A](x+y),$$
which yields
$$A(x-y)^2=0.$$
If $A\neq 0$, then $x=y$,
i.e., $\nu_1/\nu_2=\nu'_1/\nu'_2$. So $\nu'=c\nu$ and 
$\eta'=c\eta$, using Lemma~1.1, which contradicts
$$\text{dim }\{(\nu, \eta), (\nu', \eta')\}=2.$$
If $A=0$, equate the rational part of \eqref{det0} gives:
$$\aligned &(1+x^2-2Ax)(1+y^2-2Ay)+4xy\Omega^2 \\
=&(1+x^2)(1+y^2)+4xy\Omega^2\\
=&[1+xy-(x+y)A]^2+(x+y)^2\Omega^2\\
=&[1+xy]^2+(x+y)^2\Omega^2.\endaligned$$
This leads to
$$(\Omega^2-1)(x-y)^2=0.$$
Since $$\Omega^2=\omega_1^2\omega_2^2=(j_1^2+1)(j_2^2+1)>1,$$
this gives again $x=y$, which is a contradiction. This concludes the proof of \eqref{+-}
for case (B2). 

\bigskip

\noindent {\bf Case (B3)}: $|\text{supp }\nu|\geq 3$

We start from 

\noindent (i) $|\text{supp }\nu|=3$

Without loss, one may write 
$$\nu=(\nu_1, \nu_2, \nu_3, 0, .., 0), \, \nu_1\neq 0, \nu_2\neq 0, \nu_3\neq 0,$$ 
and 
$$\nu'=(\nu'_1, \nu'_2, \nu'_3, 0, .., 0), \, \nu'_3\neq 0.$$ 
Write
\begin{equation}\label{g}
g=\begin{pmatrix} \eta\cdot\eta&\eta\cdot\eta' \\
 \eta'\cdot\eta&\eta'\cdot\eta'\\
 \end{pmatrix}:=\begin{pmatrix} g_{\nu\nu}&g_{\nu\nu'} \\
g_{\nu'\nu}&g_{\nu'\nu'}\\
 \end{pmatrix},
\end{equation}
$g$ is a positive definite Euclidean metric. Denote the inverse by $g^{-1}$, which again is a positive definite  Euclidean metric.
Denote its matrix elements by
$$g^{\nu\nu}, g^{\nu\nu'}, 
g^{\nu'\nu} \text{ and } g^{\nu'\nu'}.$$
Let 
$$D= -\det T_{+-}/\det g.$$
Then simple computation gives
$$\aligned D=g^{\nu\nu}&[\nu_1^2\omega_1^2+\nu_2^2\omega_2^2+\nu_3^2\omega_3^2+2\nu_1\nu_2\omega_1\omega_2+2\nu_2\nu_3\omega_2\omega_3+2\nu_3\nu_1\omega_3\omega_1]\\
+g^{\nu'\nu'}&[{\nu'_1}^2\omega_1^2+{\nu'_2}^2\omega_2^2+{\nu'_3}^2\omega_3^2+2\nu'_1\nu'_2\omega_1\omega_2+2\nu'_2\nu'_3\omega_2\omega_3+2\nu'_3\nu'_1\omega_3\omega_1]\\
-2g^{\nu\nu'}&[\nu_1\nu'_1\omega_1^2+\nu_2\nu'_2\omega_2^2+\nu_3\nu'_3\omega_3^2+(\nu_1\nu'_2+\nu'_1\nu_2)\omega_1\omega_2+(\nu_2\nu'_3+\nu'_2\nu_3)\omega_2\omega_3+(\nu_3\nu'_1+\nu'_3\nu_1)\omega_3\omega_1]\\
-1\endaligned$$
If $D=0$, the irrational part of the equation gives the following system of three equations:
\begin{equation}\label{system}
\begin{cases}
\nu_1\nu_2g^{\nu\nu}+\nu'_1\nu'_2g^{\nu'\nu'}-(\nu_1\nu'_2+\nu'_1\nu_2)g^{\nu\nu'}=0,\\
\nu_2\nu_3g^{\nu\nu}+\nu'_2\nu'_3g^{\nu'\nu'}-(\nu_2\nu'_3+\nu'_2\nu_3)g^{\nu\nu'}=0,\\
\nu_3\nu_1g^{\nu\nu}+\nu'_3\nu'_1g^{\nu'\nu'}-(\nu_3\nu'_1+\nu'_3\nu_1)g^{\nu\nu'}=0.
\end{cases}
\end{equation}

\noindent (ia) $g^{\nu\nu'}\neq 0$

Viewing $g^{\nu\nu}$, $g^{\nu\nu'}$, $g^{\nu'\nu'}$ as the variables $x$, $y$ and $z$, say,
then $x=y=z=0$ is a solution. Denote the coefficient matrix of the system by $\mathcal V$, then
there are non-zero solutions if and only if $$\det\mathcal V=0.$$ This leads to the existence
of $\alpha$ and $\beta$, $(\alpha, \beta)\neq 0$ such that 
$$
\begin{cases}
\nu_1\nu_2=\alpha \nu_2\nu_3+\beta \nu_3\nu_1,\\
\nu'_1\nu'_2=\alpha \nu'_2\nu'_3+\beta \nu'_3\nu'_1,\\
\nu_1\nu'_2+\nu'_1\nu_2=\alpha(\nu_2\nu'_3+\nu'_2\nu_3)+\beta(\nu_3\nu'_1+\nu'_3\nu_1),
\end{cases}
$$
which leads to either  
\begin{equation}\label{1}
\nu_1/\nu_3=\nu'_1/\nu'_3,
\end{equation}
or 
\begin{equation}\label{2}
\nu_2/\nu_3=\nu'_2/\nu'_3.
\end{equation}
If $\nu'_1=0$ and  $\nu'_2=0$, this is a contradiction; else substituting
 \eqref{1} in the third equation 
in \eqref{system} leads to
$$\nu_3^2g^{\nu\nu}+{\nu'_3}^2g^{\nu'\nu'}-2\nu_3{\nu'_3}^2g^{\nu\nu'}
=(v, g^{-1}v)=0,$$
where $v:=(\nu_3, -\nu'_3)\neq 0$. Since $g^{-1}$ is a positive definite metric, as mentioned earlier, below \eqref{g},
this is a contradiction; or else  substituting \eqref{2} in the second equation 
in \eqref{system} leads to the same contradiction with $v:=(\nu_2, -\nu'_2)\neq 0$.
So $$\det T_{+-}\neq 0.$$

\noindent (ib) $g^{\nu\nu'}=0$

One may assume $|\text {supp }\nu'|=|\text {supp }\nu|=3$, otherwise clearly it is a contradiction.
Subsequently one deduces from \eqref{system} that 
$$\nu'_1/\nu_1=\nu'_2/\nu_2=\nu'_3/\nu_3.$$
So $$\nu=c\nu',$$
which yields
$$\eta=c\eta'$$
from Lemma~1.1, which contradicts 
$$\text{dim } \{\eta, \eta'\}=2.$$ So again
$$\det T_{+-}\neq 0.$$

\noindent (ii) $|\text{supp }\nu|>3$

Clearly without loss, one may again write 
$$\nu=(\nu_1, \nu_2, \nu_3, \nu_4, ..., \nu_b), \, \nu_1\neq 0, \nu_2\neq 0, \nu_3\neq 0,  \nu_4\neq 0\, ... $$ 
and 
$$\nu'=(\nu'_1, \nu'_2, \nu'_3,  \nu'_4, ..., \nu'_b), \, \nu'_3\neq 0.$$ 
The linear system obtained from the irrational part of the equation contains \eqref{system} as a sub-system. 
When $g^{\nu\nu'}\neq 0$, the same argument applies and leads to $\det T_{+-}\neq 0$. When $g^{\nu\nu'}=0$,
if $|\text {supp }\nu'|<|\text {supp }\nu|$, it is a clear contradiction, otherwise
it leads to 
$$\nu_1'/\nu_1=\nu'_2/\nu_2=\nu'_3/\nu_3=\nu'_4 /\nu_4\,  ... ,$$
and the same conclusion holds. This finishes the proof.
\end{proof}

\subsection{Proof of the New Separation Lemma}\hfill

\noindent {\it (A) Proof of lower bounds on $\det T_{++}$ and $|\det T_{+-}|$ at 
$\omega$} 

\noindent (i) The $2\times 2$ case

We begin with $T_{+-}$. We have
$$\aligned D:&=\det T_{+-}\\
&=-[\eta(\nu'\cdot\omega)^2-\eta'(\nu\cdot\omega)]^2+\eta^2{\eta'}^2-(\eta\cdot\eta')^2\\
&=P(\omega^{(0)})+\delta P_1(\omega'; \omega^{(0)})+\delta^2P_2(\omega')\\
:&=P+\mathcal P,\endaligned $$
where $P_2$ is a quadratic polynomial in $\omega'$ with integer coefficients and $P_1$ linear,
and $|\mathcal P|<\delta B^4$.  Further
$$\frac{\partial^2P_2}{\partial\omega_k^2}=-2\Vert [\eta\nu'_k-\eta'\nu_k]\Vert ^2,\, k=1, 2, ..., b.$$
So there exists $k$, such that 
$$\frac{\partial^2P_2}{\partial\omega_k^2}\neq 0,$$
otherwise it contradicts with $\text{dim }\{(\nu, \eta), (\nu'\eta')\}=2$ on the sub-module $\mathcal U$,
therefore 
\begin{equation}\label{pos}
|\frac{\partial^2P_2}{\partial\omega_k^2}|\geq 2,
\end{equation}
for some $k$. 

From Lemmas~2.1, 2.2 and the Main Lemma, 
$$P\neq 0.$$ Assume $|(\nu, \eta)|$ and $|(\nu', \eta')|\leq N$, $N\in [1, B]$. The simultaneous Diophantine
approximation result from \cite{Schm1}, see pp 151-155 \cite{Schm2}, then gives that
$$|P|\geq \frac{1}{N^{\gamma}}, \,\gamma>0.$$
Clearly $$|\mathcal P|<\delta N^4.$$
So for $N$ such that 
$$\delta N^4<\frac{1}{N^{2\gamma}},$$
$$|D|\geq \frac{1}{N^{\gamma}}-\frac{1}{N^{2\gamma}} >\frac{1}{2N^{\gamma}}.$$

At scale $N$, the number of quadratic polynomials in $\omega$ from $D$, are bounded above by 
\begin{equation}\label{number}
[N^4]^{(b+C_b^2+b+1)}<N^{4b^2}.
\end{equation}
Set 
\begin{equation}\label{q}
q=4(2\gamma+4+2b^2+1).
\end{equation}
If \begin{equation}\label{delta}
\delta N^4\geq\frac{1}{N^{2\gamma}},
\end{equation}
then 
$$|D|\geq \frac{1}{N^q},$$ holds at all  scales $N$,
away from a set in $\omega'$ of measure less than $\delta^{1+\epsilon}$, $\epsilon>0$,
by summing over $N$ and using \eqref{pos}, \eqref{number} and \eqref{delta}. Similar arguments
yields the same bound for the $T_{++}$. 

\smallskip

\noindent (ii) The $1\times 1$ case

$$\aligned D&=(\nu\cdot\omega)^2-\eta^2\\
&=(\nu\cdot\omega^{(0)})^2-\eta^2+2\delta (\nu\cdot\omega^{(0)})(\nu\cdot\omega)+\delta^2(\nu\cdot\omega')^2.\endaligned$$
Since $\nu$ has at least one non-zero component, say $\nu_1$, then
$$\frac{\partial^2 D}{\partial \omega_1^2}=2\delta^2\nu_1^2\geq 2\delta^2.$$
So we obtain the bounds on $\det T_{++}$, and likewise $\det T_{+-}$, as in the $2\times 2$ case.
 \hfill$\square$ 

\smallskip

 
   \noindent{\it (B) Proof of the New Separation Lemma}
  
\noindent The lower bounds in (i) and (ii) put us in the same setting as that of Lemma~20.14 \cite{B2}, and 
the proof there is directly applicable. (In fact it is the same.) So let us 
sketch the idea of the proof instead. 

Let $I=\{0, 1, 2, ..., \ell\}$. Using the condition \eqref{5.17}, the number of elements in the set $I$ can be bounded in terms of the 
diameter: $$\rho=\max \Vert \eta_i-\eta_k\Vert, \, i, k \in I.$$ Since 
$$\Vert \eta_i-\eta_k\Vert\leq \Vert T_+(\xi_i-\xi_k)\Vert:=\Vert T_+\zeta\Vert, $$
the idea is to bound $T_+\zeta$ by its projection onto an appropriate basis. 

Now the restriction to the hyperbolic tube: 
$$|(n_i\cdot\omega+\theta)^2-h_i^2-1|<1,$$
entails an upper bound:
$$\Vert T_+\xi_i\cdot T_-\Delta_{i'}\xi\Vert<M,$$
where $\Delta_{i'}\xi=\xi_{i'}-\xi_{i'-1}$.  Using the lower bound on $\det T_{+-}$, and the 
upper bound $M$, then leads to an upper bound:
$$\Vert T_+\zeta\Vert<\mathcal M,$$
provided the $T_-\Delta_{i'}\xi$ form a basis. Consequently 
$$\rho=\max \Vert \eta_i-\eta_k\Vert<\mathcal M.$$

Achieving this basis yields dimension reduction from $d_1\leq d$ to $d_1=1$. The upper bound
$\mathcal M$ depends on $d_1$, the lower bound on $\det T_{+-}$, and the number of $\Delta_{i}\xi$ needed, so that $T_-\Delta_{i}\xi$ form a basis.
When $d_1$ is reduced to $d_1=1$, this number is $1$, which concludes the proof.
 \hfill$\square$

\subsection{About Lemma~5.2 \cite{W2}} In \cite{W2}, we proved a separation lemma on the characteristic: 
$$(n\cdot\omega^{(0)}+\theta)^2-h^2-1=0,$$
valid in arbitrary dimensions $d$. It needs, however, the additional condition: 

\noindent (**) If $b>d+1$, any $d$ vectors in the set $\{j_k\}_{k=1}^b$ are linearly independent. 
For all $k$, $k=1, 2, ..., b$, define the set 
$$J_k=\{j_{k'}-j_k|k'=1, ..., b, k'\neq k\},$$ 
any $d$ vectors in $J_k$ are linearly independent,
(If $b\leq d$, there is no condition (**).)

For the benefit of a direct comparison with the New Separation Lemma, we rephrase it below.

\begin{lemmaw} Assume that (*) and (**) are satisfied. Let $\theta\in\mathbb R$, and set $x:=(n, h)\in\mathbb Z^{b+d}$. 
Let $\mathcal U\subset \mathbb Z^{b+d}$ be 
the sub-module generated by $u^{(0)}$. Consider a sequence, a subset  $\{x_k\}_{k=0}^\ell\subset\mathcal U\subset\mathbb Z^{b+d}$,
satisfying 
\begin{equation}
0\neq n_i-n_{i-1}=\sum_{k=1}^bm_ke_k, \, \sum_{k=1}^b|m_k|\leq p.
\end{equation}
Assume that 
for all $i$, $0\leq i\leq \ell$,
\begin{equation}\label{5.16}
(n_i\cdot\omega^{(0)}+\theta)^2-h_i^2-1=0.
\end{equation}
Then 
\begin{equation}
\ell\leq 4b,
\end{equation}
for all fixed $\theta$.
\end{lemmaw}

\begin{remark} Note that in spite  of the name, at $\omega=\omega^{(0)}$,
the lemma is valid for all scales. By perturbation, it maybe 
used at small scales for $\omega=\omega^{(0)}+\delta\omega'$, $\delta\ll1$,
hence its name.
\end{remark}


\end{document}